\documentclass[12pt
]{scrartcl}
\usepackage{amsthm,amsmath,amssymb}
\usepackage{tikz}
\usepackage{tikz-cd}
\newcommand{\isom}{\cong}
\newcommand{\set}[1]{\{#1\}}
\newcommand{\defset}[2]{\{#1\,\mid\,#2\}}
\newcommand{\gen}[1]{\mathopen{<}#1\mathclose{>}}
\newcommand{\abs}[1]{\lvert#1\rvert}

\DeclareMathOperator{\Gal}{Gal}
\newcommand{\CC}{{\mathbb{C}}}
\newcommand{\FF}{\mathbb F}            
\newcommand{\QQ}{\mathbb Q}
\newcommand{\RR}{{\mathbb R}}
\newcommand{\ZZ}{\mathbb Z}
\theoremstyle{plain}
\newtheorem{Theorem}{Theorem}[section]
\newtheorem{Lemma}[Theorem]{Lemma}

\newtheorem{Proposition}[Theorem]{Proposition}

\newtheorem{Question}[Theorem]{Question}

\theoremstyle{definition}
\newtheorem{Definition}[Theorem]{Definition}
\newtheorem{Remark}[Theorem]{Remark}
\newtheorem{Example}[Theorem]{Example}
\begin{document}
\title{Decompositions of rational functions over real and complex
  numbers and a question about invariant curves} \author{Peter
  M\"uller}
\maketitle
\begin{abstract}
  We consider the connection of functional decompositions of rational
  functions over the real and complex numbers, and a question about
  curves on a Riemann sphere which are invariant under a rational
  function.
\end{abstract}
\section{Introduction}
Let $\hat\CC=\CC\cup\set{\infty}$ be the Riemann sphere, and
$\hat\RR=\RR\cup\set{\infty}$. A \emph{circle} in $\hat\CC$ is either
a usual circle in $\CC$, or a line in $\hat\CC$. So the circles in
$\CC$ are just the curves $\Gamma=\lambda(\hat\RR)$, where
$\lambda(z)=\frac{az+b}{cz+d}\in\CC(z)$ is a linear fractional
function (with $ad-bc\ne0$).

Long ago Fatou suggested to study (Jordan) curves
$\Gamma\subset\hat\CC$ which are invariant under a rational function
of degree $\ge2$. See \cite{eremenko:invariant_curves} for recent
progress on this. The case that $\Gamma$ lies in a circle
$\lambda(\hat\RR)$ is not interesting, because any rational function
$r=\lambda^{-1}\circ s\circ\lambda$ with $s\in\RR(z)$ leaves $\Gamma$
invariant, and there are no other rational functions with this
property.

Motivated by his results on invariant curves in
\cite{eremenko:invariant_curves}, Alexandre Eremenko suggested to
investigate the following source of invariant curves, and raised two
questions about this family:
\begin{Question}
Let $f,g\in\CC(z)$ be non-constant rational functions, such that
$f(g(z))\in\RR(z)$, so the curve $\Gamma=g(\hat R)$ is invariant
under $r=g\circ f$. Assume that $\Gamma$ is not contained in a circle.
\begin{itemize}
\item[(a)] Is it possible that $\Gamma$ is a Jordan curve?
  (\cite{eremenko:invariant_curves}, \cite{eremenko:mo})
\item[(b)] Is it possible that $r:\Gamma\to\Gamma$ is injective?
  (\cite{eremenko:personal}, and special case of
  \cite{eremenko:invariant_analytic})
\end{itemize}
\end{Question}
Note that $\Gamma=g(\hat R)$ is contained in a circle if and only if
there is a linear fractional function $\lambda\in\CC(z)$ such that
$\tilde g=\lambda\circ g\in\RR(z)$. In this case $f\circ g=\tilde
f\circ\tilde g$ with $\tilde f=f\circ\lambda^{-1}\in\RR(z)$. So the
decomposition of $f\circ g$ over $\CC$ essentially arises from a
decomposition over $\RR$.

There are rational functions $f\circ g\in\RR(z)$ whose decompositions
do not come from a decomposition over the reals. On the other hand, it
is known that decompositions of real polynomials over the complex
numbers always arise from real decompositions. See Section \ref{S:RC}
for more about this.

The purpose of this paper is to give a positive answer to question
(a), and a negative answer to a slight weakening of (b). More precisely,
regarding (a), we show:
\begin{Theorem}\label{T:ell}
  For every odd prime $\ell$ there are rational functions
  $f,g\in\CC(z)$ both of degree $\ell$, such that
  \begin{itemize}
  \item[(a)] $f(g(z))\in\RR(z)$.
  \item[(b)] $g:\hat\RR\to\hat\CC$ is injective, so $g(\hat\RR)$ is a
    Jordan curve.
\item[(c)] $g(\hat\RR)$ is not a circle.
  \end{itemize}
\end{Theorem}
In order to formulate the next two results, we define a weakening of
injectivity of rational functions on $\RR$.
\begin{Definition}
  A rational function $g\in\RR(z)$ is said to be \emph{weakly
    injective on $\RR$}, if there exists $z_0\in\RR$ which is not a
  critical point of $g$, and besides $z_0$ there is no $y_0\in\hat\RR$
  with $g(z_0)=g(y_0)$.
\end{Definition}
A partial answer to question (b) is
\begin{Theorem}\label{T:b}
Let $f,g\in\CC(z)$ be non-constant rational functions, such that
$f\circ g\in\RR(z)$. Assume that $g$ is weakly injective, and that the
curve $\Gamma=g(\hat\RR)$ is not contained in a circle. Then the map
$g\circ f:\Gamma\to\Gamma$ is not injective.
\end{Theorem}
A slight variant of this theorem shows that for a fairly large class
of rational functions from $\RR(z)$, each decomposition over $\CC$
arises from a decomposition over $\RR$.
\begin{Theorem}\label{T:RC}
  Let $f,g\in\CC(z)$ be non-constant rational functions such that
  $h(z)=f(g(z))\in\RR(z)$ is weakly injective. Then there is a linear
  fractional function $\lambda\in\CC(z)$ such that $\lambda\circ
  g\in\RR(z)$.
\end{Theorem}
The main ingredient (besides Galois theory) in the proof of the
previous two theorems is the following group-theoretic result. (See
Section \ref{S:Gs} for the notation.) 
\begin{Proposition}\label{P:Gs}
  Let $G$ be a transitive group of permutations of the finite set
  $\Omega$. Let $\sigma$ be a permutation of $\Omega$ of order $2$
  which fixes exactly one element $\omega$, and which normalizes $G$,
  that is $G^\sigma=G$. Let $G_\omega$ be the stabilizer of $\omega$
  in $G$. Then $M^\sigma=M$ for each group $M$ with $G_\omega\le M\le
  G$.
\end{Proposition}
The proof of Theorem \ref{T:ell} uses elliptic curves. The
construction was motivated by a group-theoretic analysis similar to
the one which led to the proofs of Theorems \ref{T:b} and \ref{T:RC}.

\paragraph{Acknowledgment.} I thank Alexandre Eremenko for inspiring
discussions about invariant (Jordan) curves, and Theo Grundh\"ofer for
a careful reading of the somewhat intricate proof of Proposition
\ref{P:Gs}, and his suggestion which simplified its final
step. Furthermore, I thank Mike Zieve for discussions about an example
in the final section and for pointing out its connection with work by
Avanzi and Zannier.
\section{Non-circle Jordan curves invariant under a rational function}
In this section we work out our sketch from \cite{pm:mo_circles} and
prove Theorem \ref{T:ell}.

Let $E$ be an elliptic curve given by a Weierstrass equation
$Y^2=X^3+aX+b$ with $a,b\in\RR$. By $E(\CC)$ and $E(\RR)$ we denote
the complex and real points of $E$. For $p\in E(\CC)$ we let $\bar p$
be the complex conjugate of $p$. We use the structure of $E(\CC)$ as
an abelian group, with neutral element $0_E$ the unique point at
infinity.

For general facts about elliptic curves see e.g.~\cite{Silverman2}.
\begin{Lemma}
  Let $\ell\ge3$ be a prime. Then there is a point $c\in E(\CC)$ of
  order $\ell$, with $\bar c\notin\gen{c}$.
\end{Lemma}
\begin{proof}
  Let $E[\ell]\subset E(\CC)$ be the group of $\ell$-torsion
  points. Then $E[\ell]$ is isomorphic to the vector space
  $\FF_\ell^2$, and the complex conjugation acts linearly on this
  space.

  Suppose that the claim does not hold, so the complex conjugation
  fixes each $1$-dimensional subspace of $E[\ell]$ setwise. Then the
  complex conjugation acts as a scalar map. Therefore either
  $E[\ell]\subset E(\RR)$, or $\bar c=-c$ for each $c\in E[\ell]$. In
  the latter case write $c=(u,v)$. So $u$ is real and $v$ is purely
  imaginary. Thus, upon replacing $E$ with the twisted curve
  $-Y^2=X^3+aX+b$ (which is isomorphic over $\RR$ to $Y^2=X^3+aX-b$),
  we obtain in either case an elliptic curve $E$ with
  $E[\ell]\subseteq E(\RR)$. On the other hand, $E(\RR)$ is isomorphic
  to $\RR/\ZZ$ or to $\RR/\ZZ\times \ZZ/2\ZZ$ (see
  e.g.~\cite[V.~Cor.~2.3.1]{Silverman2}). However, $\RR/\ZZ\times
  \ZZ/2\ZZ$ does not have a subgroup isomorphic to
  $\ZZ/\ell\ZZ\times\ZZ/\ell\ZZ$. This proves the claim.
\end{proof}
\begin{Lemma}\label{L:w}
  Suppose that $X^3+aX+b$ has three distinct real roots. Then there
  are elements $w\in E(\RR)$ such that there is no $\hat w\in E(\RR)$
  with $w=2\hat w$.
\end{Lemma}
\begin{proof}
  If $X^3+aX+b$ has three distinct real roots, then
  $E(\RR)\isom\RR/\ZZ\times \ZZ/2\ZZ$, so any $w$ corresponding to
  $(s,1)$, $s\in\RR/\ZZ$ arbitrary, has the property that there is no
  $\hat w\in E(\RR)$ with $w=2\hat w$. (In this case, $E(\RR)$ has two
  connected components, and for each $\hat w\in E(\RR)$ the element
  $2\hat w$ is on the connected component of $0_E$.)
\end{proof}
By an \emph{automorphism} of an elliptic curve we mean a birational
map of the curve to itself which need not fix the neutral element.

Pick $c\in E(\CC)$ of order $\ell$ such that $\bar c\notin\gen{c}$,
and set $C=\gen{c}$.

Let $\Phi:E\to E'=E/C$ be the isogeny with kernel $C$. Let
$\Phi':E'\to E$ be the dual isogeny. Then $\Phi'\circ\Phi:E\to E$ is
the multiplication by $\ell$ map on $E$.

For $w\in E(\RR)$ as in the previous lemma define involutory
automorphisms
\begin{itemize}
\item $\beta$ of $E$ by $\beta(p)=w-p$,
\item $\beta'$ of $E'$ by $\beta'(p')=\Phi(w)-p'$, and
\item $\beta''$ of $E$ by $\beta''(p'')=\Phi'(\Phi(w))-p''=\ell w-p''$.
\end{itemize}
Note that $\beta'(\Phi(p))=\Phi(w)-\Phi(p)=\Phi(w-p)=\Phi(\beta(p))$,
so
\begin{equation}\label{e:betaphi}
\beta'\circ\Phi=\Phi\circ\beta\text{ and likewise
}\beta''\circ\Phi'=\Phi'\circ\beta'.
\end{equation}
In the following we show that
\begin{itemize}
\item $E/\gen{\beta}$, $E'/\gen{\beta'}$, and $E/\gen{\beta''}$
are projective lines, and
\item that there are degree $2$ branched covering maps $\psi$,
  $\psi'$, and $\psi''$ from the elliptic curves to these lines,
\item such that $\psi$ and $\psi''$ are defined over $\RR$, and
\item that after selecting uniformizing elements of the projective
  lines, there are unique rational functions $f$ and $g$ such that the
  following diagram commutes:
\end{itemize}
\[
\begin{tikzcd}
E \arrow{r}{\Phi}\arrow{d}{\Psi}\arrow[bend left]{rr}{\text{multiplication by $\ell$}}
& E' \arrow{r}{\Phi'} \arrow{d}{\Psi'}
& E \arrow{d}{\Psi''}\\
E/\gen{\beta} \arrow{r}{g} &  E'/\gen{\beta'} \arrow{r}{f} & E/\gen{\beta''} 
\end{tikzcd}
\]
We give an algebraic rather than a geometric description of the
functions $f$ and $g$. This has the advantage that the method can be
used to compute explicit examples, as we do at the end of this
section.

Let $\CC(E)$ and $\CC(E')$ be the function fields of $E$ and $E'$,
respectively. Let $x$ and $y$ be the coordinate functions with
$x(p)=u$ and $y(p)=v$ for $p=(u,v)\in E(\CC)$. So $E(\CC)=\CC(x,y)$
with $y^2=x^3+ax+b$. The comorphism $\beta^\star$ is an automorphism
of order $2$ of the real function field $\RR(E)$ (recall that $w\in
E(\RR)$). We compute the fixed field of $\beta^\star$ in $\RR(E)$:
Write $w=(w_x,w_y)$, and set $z=\frac{w_y+y}{w_x-x}$ (this choice of
$z$ is taken from \cite{macrae_samuel}). The addition formula for
elliptic curves shows that
\[
\beta^\star(x)=z^2-w_w-x\text{ and
}\beta^\star(y)=z(w_x-\beta^\star(x))-w_y.
\]
From that we get
\[
\beta^\star(z)=\frac{w_y+\beta^\star(y)}{w_x-\beta^\star(x)}=
\frac{w_y+z(w_x-\beta^\star(x))-w_y}{w_x-\beta^\star(x)}=z,
\]
so $z$ is in the fixed field of $\beta^\star$. Clearly
$\RR(x,z)=\RR(y,z)=\RR(x,y)$. Let $F\subseteq\RR(E)$ be the fixed
field of $\beta^\star$. From $z\in F$ and $[\RR(E):F]=2$ we get
$F=\RR(z)$ once we know that $[\RR(E):\RR(z)]\le 3$. But this holds,
as
\[
(z(w_x-x)-w_y)^2=y^2=x^3+ax+b,
\]
so $x$ has at most degree $3$ over $\RR(z)$. Now $\psi$ is just the
rational function $E\to\hat\CC$ for which $z=\psi(x,y)$.

Suppose (without loss of generality) that $E'$ has also a
Weierstrass from $Y^2=X^3+a'X+b'$, and let $x'$ and $y'$ be the
associated coordinate functions. Note that
$\Phi((u,v))=(A(u),B(u)v)$ for rational functions
$A,B\in\CC(z)$ and all $p=(u,v)\in E(\CC)$. Therefore
$\Phi^\star(x')=A(x)$, where $\Phi^\star:\CC(E')\to\CC(E)$ is the
comorphism of $\Phi$.

Pick $z'$ in $\CC(x',y')$ such that $\CC(z')$ is the fixed field of
$\beta'^\star$. (Here $z'$ can not be taken from $\RR(x',y')$, because
$E'$ is not defined over $\RR$.) As before, let $\psi'$ be the
rational function with $z'=\psi'(x',y')$.

From \eqref{e:betaphi} we obtain
$\Phi^\star\circ\beta'^\star=\beta^\star\circ\Phi^\star$, so
$\Phi^\star(z')=\beta^\star(\Phi^\star(z'))$ and hence
$\Phi^\star(z')\in\CC(z)$. So $\Phi^\star(z')=g(z)$ for a rational
function $g\in\CC(z)$. This is just the algebraic description of $g$
from above. Similarly one computes $f$.

For the rest of this section, we work with the morphisms rather than
the comorphisms. Recall that $\Psi$ and $\Psi''$ are defined over
$\RR$. So $\overline{\Psi(p)}=\Psi(\bar p)$ for all $p\in E(\CC)$.

We now prove the required properties of $f$ and $g$. The assertion
about the degrees follows from well-known facts about isogenies.
\begin{itemize}
\item[(a)] As the multiplication by $\ell$ map is defined over $\RR$,
  and so are $\Psi$ and $\Psi''$, we have $f\circ g\in\RR(x)$.
\item[(b)] We next show that $g$ is injective on
  $\hat\RR$. Suppose that there are distinct
  $z_1,z_2\in\hat\RR$ such that $g(z_1)=g(z_2)$. Pick $p,q\in
  E(\CC)$ such that $\Psi(p)=z_1$, $\Psi(q)=z_2$. Then
\[
\Psi'(\Phi(p))=g(\Psi(p))=g(z_1)=g(z_2)=g(\Psi'(q))=\Psi'(\Phi(q)),
\]
so $\Phi(p)=\Phi(q)$ or $\Phi(p)=\Phi(w)-\Phi(q)$. Upon possibly
replacing $q$ with $w-q$ we may and do assume $\Phi(p)=\Phi(q)$, hence
$p-q\in C$.

Recall that $\Psi$ is defined over $\RR$ and $\Psi(p)=z_1$ is real. So
$\Psi(\bar p)=\Psi(p)$, and therefore $\bar p=p$ or $\bar
p=w-p$. Likewise $\bar q=q$ or $\bar q=w-q$. Recall that $p-q\in C$,
and that $C\cap\bar C=\set{0_E}$ by the choice of $C$. So we can't
have $(\bar p,\bar q)=(p,q)$, nor $(\bar p,\bar q)=(w-p,w-q)$.

Thus, without loss of generality, $\bar p=p$ and $\bar q=w-q$. So
$p\in E(\RR)$. Note that $p-q$ and $\bar p-\bar q=p-w+q$ both have
order $\ell$. Set $r=(p-q)+(\bar p-\bar q)=2p-w$. Then $\ell r=0_E$,
and $r\in E(\RR)$. We obtain $w=2(p+\frac{\ell-1}{2}r)$ with
$p+\frac{\ell-1}{2}r\in E(\RR)$, contrary to the choice of $w$.
\item[(c)] Finally, we need to show that $g(\hat\RR)$ is
  not a circle. Suppose otherwise. Let $\lambda$ be a linear
  fractional function which maps this circle to
  $\hat\RR$. Then $\lambda\circ g$ maps $\RR$ to
  $\hat\RR$, so $\lambda\circ g\in\RR(x)$.

Then $\lambda\circ\Psi'\circ\Phi=\lambda\circ g\circ\Psi$ is defined
over $\RR$, so $\lambda(\Psi'(\Phi(p)))=\lambda(\Psi'(\Phi(\bar p)))$
for all $p\in E(\CC)$. As $\lambda$ is bijective, $\Psi'$ respects
$\beta'$, and $C$ is the kernel of $\Phi$, we get that for each $p\in
E(\CC)$ either $p-\bar p\in C$, or $p+\bar p-w\in C$.

In the first case note that $\bar p-p\in\bar C$, so also $p-\bar
p=-(\bar p-p)\in\bar C$, and therefore $p-\bar p\in C\cap\bar
C=\set{0_E}$. So $p\in E(\RR)$ if the first case happens.

We see that $p+\bar p-w\in C$ whenever $p\in E(\CC)\setminus
E(\RR)$. Recall that $w\in E(\RR)$. So $p+\bar p-w\in C\cap\bar
C=\set{0_E}$, and therefore $p+\bar p=w$ for all $p\in E(\CC)\setminus
E(\RR)$. As $(p+q)+\overline{p+q}=2w\ne w$ for all $p,q\in
E(\CC)\setminus E(\RR)$, we get the absurd consequence that $p+q\in
E(\RR)$ whenever $p,q\in E(\CC)\setminus E(\RR)$, so $E(\RR)$ is a
subgroup of index $2$ in $E(\CC)$. This final contradiction proves all
the properties about the functions $f$ and $g$.
\end{itemize}
\begin{Remark}
  For fixed curves $E$, $E'$ and isogeny $\Phi$ as in the proof of
  Theorem \ref{T:ell}, and $w\in E(\RR)$ (which need not fulfill the
  property of Lemma \ref{L:w}), let $h_w=f\circ g\in\RR(z)$ be the
  rational function constructed there. The case $w=0_E$ gives a
  Latt\`es function $h_0\in\RR(z)$. It is easy to see that
  $h_w=\lambda_1\circ h_0\circ\lambda_2$ for linear fractional
  functions $\lambda_1,\lambda_2\in\CC(z)$. If $w$ has the property
  from Lemma \ref{L:w}, then $\lambda_1,\lambda_2$ cannot be chosen in
  $\RR(z)$. So $h_w$ is a twist of $h_0$ over a quadratic
  field. Therefore, the relation of $h_w$ to the Latt\`es map $h_0$ is
  analogous to the relation of R\'edei functions to cyclic polynomials
  $z^n$.

A construction like $h_w$ appeared in an arithmetic context in
\cite{GMS}.

Latt\`es functions, which were known before Latt\`es work in 1918, are
classical objects in complex analysis. See \cite{silverman:dynamical}
and \cite{milnor:lattes} for the relevance of these functions in
complex dynamics, and especially \cite{milnor:lattes} for a lot of
information about the history of these functions. Latt\`es functions
also appeared in 1877 in the context of approximation theory in work
by Zolotarev. Today they are called Zolotarev functions in
approximation theory.
\end{Remark}
\begin{Example}
  Here we explicitly compute an example for the case $\ell=3$. We aim
  to find an example where the elliptic curve $E$ is defined over
  $\QQ$, $f\circ g \in\QQ(z)$, and $f,g\in K(z)$, where $K$ is an as
  small as possible number field. Let $\omega$ be a primitive third
  root of unity, so $\omega^2+\omega+1$ and $\bar\omega=-1-\omega$. As
  $c\in E(\CC)$ is required to be a non-real point, and the
  coordinates of the $\ell$-torsion group of an elliptic curve over
  $\QQ$ generate the field of $\ell$-th roots of unity, we necessarily
  have $\omega\in K$. Indeed, there are examples with $K=\QQ(\omega)$.

  The in terms of the conductor smallest elliptic curve $E$ over
  $\QQ$ which has a $3$-torsion point in $E(K)\setminus E(\QQ)$ has
  the Cremona label 14a2 and Weierstrass form
  $Y^2=X^3-46035X-3116178$. One computes that $c=(72\omega - 33,
  1080\omega - 648)\in E(\CC)$ has order $3$. Set $C=\gen{c}$. Then
  $C\cap\bar C=\set{0_E}$. There is an isogeny $\Phi:E\to E'$ with
  kernel $C$, where $E'$ is given by
  $Y^2=X^3+(298080\omega+537165)X+(86819040\omega-39204594)$.

  Set $w=(-78,0)\in E(\QQ)$. The $X$-coordinates of $\hat w$ with
  $2\hat w=w$ are roots of $X^2 + 156X + 33867=(X+78)^2+27783$, so
  there is no $\hat w\in E(\RR)$ with $2\hat w=w$.

  Thus $E$, $C$ and $w$ fulfill all the assumptions which we needed in
  the existence proof of $f(z)$ and $g(z)$. We now compute these
  functions. Let $\beta$ and $\beta'$ be the automorphisms of $E$ and
  $E'$ given by $\beta(p)=w-p$ and $\beta'(p')=\Phi(w)-p'$. Write
  $w=(w_x,w_y)$ and $\Phi(w)=(w_x',w_y')$. Set $z=\frac{w_y+y}{w_x-x}$
  and $z'=\frac{w_y'+y'}{w_x'-x'}$. Recall that $\Phi^\star(x')=A(x)$,
  where $\Phi((u,v))=(A(u),B(u)v)$ for all $(u,v)\in E(\CC)$. From
  that we see also $\Phi^\star(y')=B(x)y$.

Now recall that the function $g(z)$ we are looking for fulfills
$g(z)=\Phi^\star(z')$. We compute
\[
g(z)=\Phi^\star\left(\frac{w_y'+y'}{w_x'-x'}\right)=\frac{w_y'+B(x)y}{w_x'-A(x)}.
\]
Use this equation, and the equations $z=\frac{w_y+y}{w_x-x}$ and
$y^2=x^3-46035x-3116178$ to eliminate the variables $x$ and $y$. So we
are left with a polynomial equation in $z$ and the unknown function
$g(z)$ which we treat as a variable. This polynomial has a factor of
degree $1$ with respect to $g(z)$, from which we obtain
$g(z)$. Analogously we get $f(z)$. After minor linear changes over
$\QQ$ (which slightly simplify $f$ and $g$) we obtain
\begin{align*}
f(z) &= \frac{z^3 - 6(\omega + 1)z}{3z^2 + 1}\\
g(z) &= \frac{2z^3 + (\omega + 1)z}{z^2 - \omega}\\
f(g(z)) &=  \frac{8z^9 - 24z^5 - 13z^3 - 6z}{12z^8 + 13z^6 + 12z^4
- 1}\in\QQ(z).
\end{align*}
The following plot shows the image of $\hat\RR$ under
$\frac{1}{1+g(z)}$. As expected, this curve is a Jordan curve, but not
a circle.
\includegraphics[width=.3\textwidth]{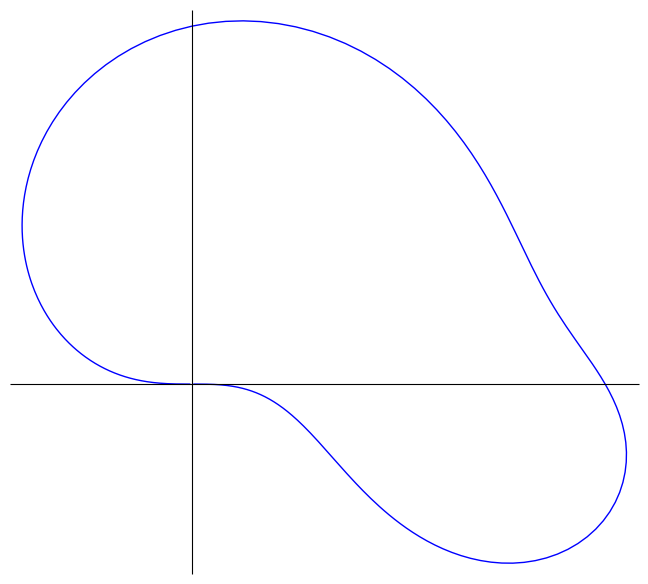}
\end{Example}
\section{Proof of Proposition \ref{P:Gs}}\label{S:Gs}
If $G$ acts on a set $\Omega$, then $\omega^g$ denotes the image of
$\omega\in\Omega$ under $g\in G$. Furthermore, $G_\omega=\defset{g\in
  G}{\omega^g=\omega}$ is the stabilizer of $\omega$ in $G$.

For $g,h\in G$ we write $g^h$ for the conjugate $h^{-1}gh$ of $g$
under $h$. Similarly, if $S$ is a subset or subgroup of $G$, then
$S^h=\defset{s^h}{s\in S}$.

If $G$ is transitive on $\Omega$, then
$\emptyset\ne\Delta\subseteq\Omega$ is called a block if
$\Delta=\Delta^g$ or $\Delta\cap\Delta^g=\emptyset$ for each $g\in
G$. If this is the case, then $\Omega$ is a disjoint union of sets
$\Delta_i=\Delta^{g_i}$ for $g_i$ in a subset of $G$. These sets
$\Delta_i$ are called a block system. Note that $G$ acts by permuting
these sets $\Delta_i$.

We assume that Proposition \ref{P:Gs} is false, so there is a group
$M$ with $G_\omega\le M\le G$ and $M\ne M^\sigma$. Among the
counterexamples with $\abs{G}$ minimal we pick one with $\abs{\Omega}$
minimal. In a series of lemmas we derive properties of such a
potential counterexample, and eventually we will see that it does not
exist.

Note that $G_\omega^\sigma$ fixes $\omega^\sigma=\omega$, hence
$G_\omega^\sigma\le G_\omega$ and therefore
$G_\omega=G_\omega^\sigma$, a fact we will use frequently. Another
trivial fact which we use throughout the proof is the following: If
$B$ is a subgroup of $G$, then $\sigma$ normalizes $B\cap B^\sigma$
and $\gen{B,B^\sigma}$.
\begin{Lemma}
  $G$ is transitive on $\Omega$.
\end{Lemma}
\begin{proof}
  Set $\Delta=\omega^G$. If $\Delta=\Omega$ then we are done. So
  assume that $\Delta\subsetneq\Omega$. If $\Delta=\set{\omega}$, then
  $G_\omega=G$, and therefore of course $M=G=M^\sigma$.

Thus $\set{\omega}\subsetneq\Delta\subsetneq\Omega$. Note that
$\Delta^\sigma=\omega^{G\sigma}=\omega^{\sigma G}=\omega^G=\Delta$. By
the assumption of a minimal counterexample, we obtain that the
proposition holds for the action of $\gen{G,\sigma}$ on $\Delta$,
hence $M=M^\sigma$, a contradiction.
\end{proof}
\begin{Lemma}\label{L:MMs}
$G=\gen{M,M^\sigma}$.
\end{Lemma}
\begin{proof}
  Set $H=\gen{M,M^\sigma}$. Then $H^\sigma=H$, so if $H$ is a proper
  subgroup of $G$, then $M=M^\sigma$ by the minimality assumption of a
  counterexample.
\end{proof}
\begin{Lemma}
$M\cap M^\sigma=G_\omega$.
\end{Lemma}
\begin{proof}
  Set $W=M\cap M^\sigma$. Note that $G_\omega\le W$ and
  $W^\sigma=W$. Therefore $\Delta=\omega^W$ is a block for the action
  of $\gen{G,\sigma}$ on $\Omega$, and $\Delta^\sigma=\Delta$. Let
  $\bar\Omega$ be the block system which contains $\Delta$, so
  $\gen{G,\sigma}$ acts on $\bar\Omega$. By the transitivity of $G$
  all blocks in $\bar\Omega$ have the same size, and this size divides
  the odd number $\abs{\Omega}$. So the blocks have odd size,
  therefore $\sigma$ has a fixed point in each block which is fixed
  setwise. Thus $\Delta$ is the only block fixed by $\sigma$. For
  $g\in G$ let $\bar g$ be the induced permutation on
  $\bar\Omega$. The stabilizer of $\Delta$ in $\bar G$ is $\bar W$.

  Now suppose that $W>G_\omega$, hence $\abs{\Delta}>1$ and therefore
  $\abs{\bar\Omega}<\abs{\Omega}$. Note that $\bar W\le\bar M$. So the
  proposition applies and yields $\bar M^\sigma=\bar M$. But the
  kernel of the map $g\mapsto\bar g$ is contained in $W=M\cap
  M^\sigma$, so $M^\sigma=M$, a contradiction.
\end{proof}
\begin{Lemma}\label{L:crucial}
If $B\le M$, then either $G=\gen{B,B^\sigma}$, or $B\le G_\omega$.
\end{Lemma}
\begin{proof}
  Set $H=\gen{B,B^\sigma}$, and suppose that $H<G$. Hence the
  proposition holds for $H$, in particular
  $\gen{B,H_\omega}^\sigma=\gen{B,H_\omega}$. Thus
  $B\le\gen{B,H_\omega}^\sigma\le\gen{M,G_\omega}^\sigma=M^\sigma$. Together
  with the previous Lemma we get $B\le M\cap M^\sigma=G_\omega$.
\end{proof}

\begin{Lemma}
$G$ has even order.
\end{Lemma}
\begin{proof}
Suppose that the order of $G$ is odd. Pick $g\in M\setminus G_\omega$. Note that
\[
(gg^{-\sigma})^\sigma=g^\sigma g^{-1}=(gg^{-\sigma})^{-1},
\]
so $\sigma$ acts on $\gen{gg^{-\sigma}}$ by inverting the elements. As
$\gen{gg^{-\sigma}}$ has odd order, there is $h\in\gen{gg^{-\sigma}}$
with $gg^{-\sigma}=h^2$. Set $c=hg^\sigma$. First note that $c$ is
fixed under $\sigma$:
\[
c^\sigma=(hg^\sigma)^\sigma=h^\sigma
g=h^{-1}g=h^{-1}h^2g^\sigma=hg^\sigma=c
\]
From this we obtain that $c$ permutes the fixed points of $\sigma$, so
$c\in G_\omega$ because $\omega$ is the only fixed point of $\sigma$.

Another calculation shows
\[
cg^{-\sigma}c=hg^\sigma g^{-\sigma}hg^\sigma=h^2g^\sigma=g,
\]
hence
\[
g\in\gen{G_\omega,g^{-\sigma}}\le M^\sigma,
\]
contrary to $M\cap M^\sigma=G_\omega$ and the choice of $g$.
\end{proof}
\begin{Lemma}
$M$ contains at least one involution which is not contained in $G_\omega$.
\end{Lemma}
\begin{proof}
  As $[G:G_\omega]=\abs{\Omega}$ is odd, there is a Sylow $2$--subgroup $S$
  of $G$ contained in $G_\omega$. By the previous lemma, $\abs{S}>1$. As the
  number of Sylow $2$--subgroups of $G_\omega$ is odd and $G_\omega=G_\omega^\sigma$, we
  may and do pick $S$ with $S^\sigma=S$.

Let $S_2$ be the set of involutions in $S$. Then
$S_2^\sigma=S_2$. Suppose that the lemma is false. Then
$S_2^m\subseteq G_\omega$ for all $m\in M$. As $\gen{S_2^m}=\gen{S_2}^m$ is a
$2$--group in $G_\omega$, there is $u\in G_\omega$ such that $S_2^{mu}\subseteq S$,
and therefore $S_2^{mu}=S_2$. So for each $m\in M$ there is $u\in G_\omega$
such that $mu\in N_M(S_2)$, where $N_M(S_2)$ denotes the normalizer of
$S_2$ in $M$. Therefore
\begin{equation}\label{e:1}
M=\gen{N_M(S_2),G_\omega}.
\end{equation}
From $S_2^\sigma=S_2$ we obtain
\begin{equation}\label{e:2}
N_M(S_2)^\sigma=N_{M^\sigma}(S_2).
\end{equation}
Set $H=\gen{N_M(S_2),N_M(S_2)^\sigma}$. If $H<G$, then $N_M(S_2)\le
G_\omega$ by Lemma \ref{L:crucial}, and therefore $M\le G_\omega$ by
\eqref{e:1}, a contradiction.

Therefore $H=G$. Together with \eqref{e:2} we obtain
\[
G=\gen{N_M(S_2),N_M(S_2)^\sigma}=\gen{N_M(S_2),N_{M^\sigma}(S_2)},
\]
so all of $G$ normalizes $S_2$. Then $Q=\gen{S_2}$ is a nontrivial
normal subgroup of $G$ with $Q\subseteq G_\omega$, so $Q$ fixes every point
in $\Omega$, a contradiction.
\end{proof}
We now obtain the final contradiction: Let $a\in M\setminus G_\omega$
be an involution. Set $b=a^\sigma\in M^\sigma$ and let $D$ be the
dihedral group generated by $a$ and $b$. From Lemma \ref{L:crucial},
with $B=\gen{a}$, we get $D=G$.

Set $C=\gen{ab}$. Then $[G:C]=2$ (because $G=C\cup Ca$). We claim that
$C$ is transitive on $\Omega$. If this were not the case, then, by the
transitivity of $G$ and $C\triangleleft G$, $C$ would have exactly two
orbits of equal size, so $\abs{\Omega}$ were even, a contradiction.

So $G=C G_\omega$, and $C\cap G_\omega=1$, because transitive abelian
groups act regularly. The modular law yields $M=(C\cap M)G_\omega$ and
$M^\sigma=(C\cap M^\sigma)G_\omega$.

From $\abs{M}=\abs{M^\sigma}$ we get $\abs{C\cap M}=\abs{C\cap
  M^\sigma}$. But the subgroups of the cyclic group $C$ are determined
uniquely by their order, hence $C\cap M=C\cap M^\sigma$ and finally
$M=M^\sigma$.
\section{Proof of Theorems \ref{T:b} and \ref{T:RC}}
For the rational function $g(z)\in\CC(z)$ let $\bar g(z)$ be the
function with complex conjugate coefficients. Recall that $g(\hat\RR)$
is a circle in $\hat\CC$ if and only if there is a linear fractional
function $\lambda\in\CC(z)$ such that $\lambda\circ
g\in\RR(z)$. The following lemma gives a useful necessary and
sufficient criterion for this to hold. By $\CC(g(z))$ we mean the
field of rational functions in $g(z)$.
\begin{Lemma}\label{L:cocycle}
  Let $g(z)\in\CC(z)$. Then $\lambda\circ g\in\RR(z)$ for some linear
  fractional function $\lambda\in\CC(z)$ if and only if
  $\CC(g(z))=\CC(\bar g(z))$.
\end{Lemma}
\begin{proof}
  If $\lambda\circ g\in\RR(z)$, then $\lambda\circ
  g=\overline{\lambda\circ g}=\bar\lambda\circ\bar g$, hence $\bar
  g(z)=\bar\lambda^{-1}(\lambda(g(z)))$, and therefore
  $\CC(\bar g(z))=\CC(g(z))$.

  To prove the other direction, suppose that $\CC(g(z))=\CC(\bar
  g(z))$. This assumption is preserved upon replacing $g$ with
  $\mu\circ g$ for a linear fractional function $\mu\in\CC(z)$. Thus,
  without loss of generality, we may assume that there are
  $r_1,r_2,r_3\in\RR$ with $g(r_1)=\infty$, $g(r_2)=0$,
  $g(r_3)=1$. From $\CC(g(z))=\CC(\bar g(z))$ we get $\bar
  g=\rho\circ g$ for a linear fractional function
  $\rho\in\CC(z)$. Evaluating in $r_1$, $r_2$, and $r_3$ yields that
  $\rho$ fizes $\infty$, $0$ and $1$, hence $\rho(z)=z$. So $\bar
  g=g$, and therefore $g\in\RR(z)$.
\end{proof}
\begin{Remark}
  The lemma holds more generally if we replace $\RR$ with a field $K$
  and $\CC$ with a Galois extension $E$ of $K$, and
  $\CC(g(z))=\CC(\bar g(z))$ by the condition
  $E(g(z))=E(g^\sigma(z))$ for all $\sigma\in\Gal(E/K)$. Indeed, if
  $K$ is an infinite field, then we find $r_1,r_2,r_3\in K$ such that
  the values $g(r_1)$, $g(r_2)$, and $g(r_3)$ are distinct and
  therefore without loss of generality equal to $\infty$, $0$ and
  $1$. So, as above, $g=g^\sigma$ for all $\sigma\in\Gal(E/K)$. Thus
  the coefficients of $g$ are fixed under $\Gal(E/K)$ and therefore
  contained in $K$.

  If $K$ is finite, we can argue as follows: We may assume that
  $g(\infty)=\infty$, so $g(z)=p(z)/q(z)$ for relatively prime
  polynomials $p,q\in E[z]$ with $\deg p>\deg q$. In addition, we may
  assume that $p$ and $q$ are monic. Let $\sigma$ be a generator of
  the cyclic group $\Gal(E/K)$. From $g^\sigma(z)\in E(g(z))$ and
  $g^\sigma(z)=\frac{p^\sigma(z)}{q^\sigma(z)}$ we obtain
  $g^\sigma=g+b$ for some $b\in E$. Repeated application of $\sigma$
  shows that $\text{Trace}_{E/K}b=0$. So by the additive Hilbert's
  Theorem 90 there is $c\in E$ with $c-c^\sigma=b$, hence
  $(g+c)^\sigma=g+c$ and therefore $g+c\in K(z)$. (The same argument,
  except that Hilbert's Theorem 90 is a trivial fact for the extension
  $\CC/\RR$, works as an alternative proof of the lemma too.)
\end{Remark}

Theorem \ref{T:b} is a direct consequence of Theorem \ref{T:RC}. For
if $g$ is weakly injective, and $g\circ f$ is injective on
$g(\hat\RR)$, then $g\circ f\circ g$ is weakly injective, so $f\circ
g$ is weakly injective even more.

Thus we only need to prove Theorem \ref{T:RC}.

Let $t$ be a variable over $\CC$, and $Z$ be another variable over the
field $\CC(t)$ of rational functions in $t$.

If $h(\infty)\ne\infty$, then upon replacing $h(Z)$ with
$\frac{1}{h(Z)-h(\infty)}$ (and $\tau$ with
$\frac{1}{\tau-h(\infty)}$) we may assume that $h(\infty)=\infty$. By
further replacing $h(Z)$ with $h(Z)-\tau$, we may and do assume that
$h(Z)=\frac{p(Z)}{q(Z)}$, where $p(Z),q(Z)\in\RR[Z]$ are relatively
prime polynomials, $\deg p(Z)=n>\deg q(Z)$, and
$p(Z)=\prod_{i=1}^n(Z-\alpha_i)$, where the $\alpha_i$ are pairwise
distinct, $\alpha_1\in\RR$, and $\alpha_i\notin\CC\setminus\RR$ for
$i\ge2$.

By Hensel's Lemma, $p(Z)-tq(Z)=\prod_{i=1}^n(Z-z_i)$, where
$z_i\in\CC[[t]]$ has constant term $\alpha_i$. As $\alpha_1\in\RR$ and
$p,q\in\RR[Z]$, we actually have $z_1\in\RR[[t]]$. Write $z=z_1$. The
complex conjugation acts on the coefficients of the formal Laurent
series $\CC((t))$ and fixes $t$. Under this action, $z$ is fixed, and
the $z_i'$'s for $i\ge2$ are flipped in pairs. Note that $t=h(z_i)$
for all $i$.

$\CC(z_1,z_2,\dots,z_n)$ is a Galois extension of $\RR(t)$, and
$\CC(z_1,z_2,\dots,z_n)\subseteq\CC((t))$. So the restriction of the
complex conjugation action of $\CC((t))$ to $\CC(z_1,z_2,\dots,z_n)$
is an involution $\sigma\in A:=\Gal(\CC(z_1,z_2,\dots,z_n)/\RR(t))$
which fixes $z=z_1$, and moves all $z_i$ with $i>1$.

Now write $h(z)=f(g(z))$ as in Theorem \ref{T:RC}. Then also
$t=h(z)=f(g(z))=\bar f(\bar g(z))$. This yields the following
inclusion of fields and the corresponding subgroups of $A$ by Galois
correspondence:
\begin{equation*}
\begin{tikzcd}[every arrow/.append style={-}]
{}  & \CC(z_1,\dots,z_n) \arrow{d} &  &  & 1 \arrow{d} &  \\
  & \CC(z) \arrow{ld} \arrow{rd} &  &  & G_z
  \arrow{ld} \arrow{rd} & \\
  \CC(g(z)) \arrow{rd} & & \CC(\bar g(z)) \arrow{ld} & M
  \arrow{rd} & &
  M^\sigma \arrow{ld}\\
  & \CC(t) \arrow{d} &  &  & G \arrow{d} &  \\
  & \RR(t) & & & A &
\end{tikzcd}
\end{equation*}
Here $G_z$ is the stabilizer of $z=z_1$ in $G$. As $M$ is the
stabilizer of $g(z)$ in $G$, and $\sigma$ maps $g(z)$ to $\bar g(z)$,
the stabilizer of $\bar g(z)$ is $\sigma^{-1}M\sigma=M^\sigma$.

By construction, $\CC(z_1,z_2,\dots,z_n)$ is the splitting field of
$p(Z)-tq(Z)$ over $\CC(t)$, hence $G$ acts faithfully on
$\set{z=z_1,z_2,\dots,z_n}$. Now $G$ is normal in $A$, so $\sigma\in
A$ normalizes $G$. Furthermore, $\sigma$ fixes exactly one of the
$z_i$. So $M=M^\sigma$ by Proposition \ref{P:Gs}. Thus
$\CC(g(z))=\CC(\bar g(z))$ by the Galois correspondence, and finally
$\lambda(g(z))\in\RR(z)$ for some linear fractional $\lambda\in\CC(z)$
by Lemma \ref{L:cocycle}. This proves Theorem \ref{T:RC}.
\section{Some more examples}\label{S:RC}
If $h=f\circ g$ for polynomials $f,g\in\CC[z]$, and $h\in\RR[z]$, then
it is well known that there is a linear polynomial $\lambda\in\CC[z]$
such that $\lambda\circ g\in\RR[z]$. See \cite[Theorem
3.5]{FriedMacRae}, or \cite[Prop.~2.2]{Turnwald:Schur} for a down to
earth proof. A less elementary but more conceptual proof can be based
on the fact that the Galois group of $h(Z)-t$ over $\CC(t)$ contains
an element which cyclically permutes the roots of $h(Z)-t$, and the
other fact that subgroups of cyclic groups are uniquely determined by
their orders.

Note that if $h=f\circ g$ for a polynomial $h$ and rational functions
$f,g$, then there is a linear fractional function $\rho\in\CC(z)$ such
that $h=(f\circ\rho^{-1})\circ(\rho\circ g)$, and $f\circ\rho^{-1}$
and $\rho\circ g$ are polynomials. (This follows from looking at the
fiber $h^{-1}(\infty)$.)

So in order to get examples of rational functions $h\in\RR(z)$ which
decompose as $h=f\circ g$ with $f,g\in\CC(z)$ such that there is no
linear fractional function $\lambda\in\CC(z)$ with $\lambda\circ
g\in\RR(z)$, one has to assume that $h$ is not a polynomial.

One also has to assume that $g$ is not a polynomial, as the following
easy result shows.
\begin{Lemma}
  Suppose that $f\circ g\in\RR(z)$ where $f\in\CC(z)$ and $g\in\CC[z]$
  are not constant. Then $\lambda\circ g\in\RR[z]$ for a linear
  polynomial $\lambda\in\CC[z]$.
\end{Lemma}
\begin{proof}
  Assume without loss of generality that $g$ is monic. Write
  $f=\frac{p}{q}$ with $p,q\in\CC[z]$ relatively prime and $p$
  monic. From $f\circ g\in\RR(z)$ we obtain
\[
\frac{\bar p(\bar g(z))}{\bar q(\bar g(z))}=\frac{p(g(z))}{q(g(z))}.
\]
Clearly, both fractions are reduced, and the numerators of both sides
are monic. Therefore $\bar p(\bar g(z))=p(g(z))$, hence $p\circ
g\in\RR[z]$, and the claim follows from the polynomial case.
\end{proof}

Now we give some examples of rational functions $h\in\RR(z)$ with a
decomposition $h=f\circ g$ with $f,g\in\CC(z)$ which is not equivalent
to a decomposition over $\RR$. Recall that this is equivalent to
$g(\hat\RR)$ not being a circle. Of course, Theorem \ref{T:ell} gives
many example for this. But these examples are quite complicated and
not explicit. However, if one drops the requirement that $g(\hat\RR)$
is a Jordan curve, then there are quite simple examples. We give two
series.
\begin{Example}[Attributed to Pakovich by Eremenko in
  \cite{eremenko:invariant_curves}]
  Let $T_n\in\RR[z]$ be the polynomial with
  $T_n(z+\frac{1}{z})=z^n+\frac{1}{z^n}$ (so $T_n$ is essentially a
  Chebychev polynomial.) Set $g(z)=\zeta z+\frac{1}{\zeta z}$ for an
  $n$-th root of unity $\zeta$. Then
  $T_n(g(z))=z^n+\frac{1}{z^n}\in\RR(z)$, while $g(\hat\RR)$ is not a
  circle if $\zeta^4\ne1$.
\end{Example}
\begin{Example} Pick $\zeta\in\CC$ with $\abs{\zeta}=1$, and set
\[
F=z^k(1-z)^{n-k},\;G=\frac{1-\zeta z^k}{1-\zeta z^n},\;\mu(z)=\frac{z+i}{z-i}
\]
for $1\le k<n$. A straightforward calculation shows that
\begin{equation}\label{e:FG}
F(\bar G(\frac{1}{z})=\frac{1}{\zeta^{n-k}}F(G(z)).
\end{equation}
Pick $\rho\in\CC$ with $\rho^2=\frac{1}{\zeta^{n-k}}$, and set
\[
f(z)=\rho F(z),\;g(z)=G(\mu(z)).
\]
From $\bar\mu(z)=\mu(\frac{1}{z})$ and \eqref{e:FG} we get
$\overline{f\circ g}=f\circ g$, hence $f\circ g\in\RR(z)$.

On the other hand, it is easy to see that, except for a some
degenerate cases, $g(\hat\RR)$ is not a circle. Furthermore, we see
that $g(\hat\RR)$ isn't even a Jordan curve (unless it is a circle),
for if $z$ runs through $\hat\RR$, $\mu(z)$ runs through the unit
circle, so the numerator and denominator of $g(z)=G(\mu(z))$ vanish
$k$ and $n$ times, respectively, so $g(\hat\RR)$ has several self
intersections.

Originally I had only found the cases $\zeta=-1$, $k=n-1$. Mike Zieve
observed the strong similarity of these examples with functions which
turned up in work of Avanzi and Zannier. In
\cite{avanzi_zannier:fx=fy} they classify triples $F\in\CC[z]$,
$G_1,G_2\in\CC(z)$ such that $F\circ G_1=F\circ G_2$. One of their
cases (\cite[Prop.~4.7(3)]{avanzi_zannier:fx=fy}) is the above series
with $\zeta=1$, and the series
\cite[Prop.~5.6(4)]{avanzi_zannier:fx=fy} is essentially our series
from above.

The connection with the work by Avanzi and Zannier is not a surprise:
If we look for polynomials $F\in\RR[z]$ such that there is
$G\in\CC(z)\setminus\RR(z)$ with $F\circ G\in\RR(z)$, then $F\circ\bar
G=F\circ G$ with $G\ne\bar G$. Furthermore, note that if $\zeta$ is an
$m$-th root of unity, then $f(z)^{2m}\in\RR(z)$. So upon setting
$\tilde f=f^{2m}=F^{2m}$, we have $\tilde f\circ g=\tilde f\circ\bar
g$.
\end{Example}

\noindent{\scshape Institut f\"ur Mathematik, Universit\"at W\"urzburg,
  97074 W\"urzburg, Germany}\par
\noindent{\slshape E-mail: }{\ttfamily peter.mueller@mathematik.uni-wuerzburg.de}
\end{document}